\documentclass{amsart}
\usepackage{colonequals}

\usepackage{booktabs}
\usepackage{color}

\theoremstyle{cupthm}
\newtheorem{theorem}{Theorem}[section]
\newtheorem{proposition}[theorem]{Proposition}

\theoremstyle{cupdefn}

\theoremstyle{cuprem}

\numberwithin{equation}{section}

\newtheorem{question}[theorem]{Question}

\newcommand{\Z}{{\mathbb Z}}

\newcommand{\Q}{{\mathbb Q}}

\begin{document}
\title{On Properly $\theta$-Congruent Numbers Over Real Number Fields}
\author{Sajad Salami}
\address{Institute of Mathematics and Statistics, Rio de Janeiro State University, Rio de Janeiro, RJ, Brazil}
\email{sajad.salami@ime.uerj.br}

\author{Arman Shamsi Zargar}
\address{Department of Mathematics and Applications, University of Mohaghegh Ardabili, Ardabil, Iran}
\email{zargar@uma.ac.ir}
\date{}




\begin{abstract}
The notion of $\theta$-congruent numbers  generalizes the classical congruent number problem.  Recall that a  positive integer $n$ is $\theta$-congruent if it is the area of a rational triangle with an angle $\theta$ whose cosine is rational. Das and Saikia~\cite{DasSaikia} established criteria for numbers to be $\theta$-congruent over certain real number fields and concluded their work by posing four open questions regarding the relationship between $\theta$-congruent and properly $\theta$-congruent numbers.  In this work, we provide  complete answers to those questions. Indeed, we remove a technical assumption from their result  on fields with degrees coprime to $6$, provide a definitive answer for real cubic fields without congruence restrictions, extend the analysis to fields of degree~$6$, and  examine  the exceptional cases 
$n=1, 2, 3$ and $6$.

\bigskip

\noindent	{\bf AMS subjects (2020):} primary 11G05; secondary 11R21, 11R16

\bigskip


\end{abstract}

\maketitle

\section{Introduction}

A positive integer $n$ is a \textit{congruent number} if it is the area of a right triangle with rational sides. This classical problem is deeply connected to the arithmetic of the elliptic  curves  $E_n: y^2 = x(x^2 - n^2)$.  It is well-known that  $n$ is congruent if and only if the Mordell--Weil group $E_n(\Q)$ has positive rank~\cite{Koblitz}.

The notion of congruent numbers has been generalized to $\theta$-congruent numbers, by Fujiwara,  as follows. Let $\theta \in (0, \pi)$ be an angle with $\cos\theta = s/r \in \Q$ for coprime integers $s, r$. A positive integer $n$ is called \textit{$\theta$-congruent} if there exists a triangle with rational sides $u,v,w$ having an angle $\theta$ between the sides $u$ and $v$, and area $n\sqrt{r^2-s^2}$. This problem is equivalent to studying the elliptic curves
\[ E_{n,\theta}: y^2 = x(x+(r+s)n)(x-(r-s)n). \]
In particular, a  number $n$ is $\theta$-congruent if and only if $E_{n,\theta}(\Q)$ has a rational point of order greater than $2$,  see \cite{Fujiwara} for more details.

Das and Saikia \cite{DasSaikia} extended this problem to number fields. A number $n$ is \textit{$(K, \theta)$-congruent} if there exists a triangle with sides $u,v,w$ in a number field $K$ having an angle $\theta$ between $u$ and $v$, and area $n\sqrt{r^2-s^2}$. It is \textit{properly $(K, \theta)$-congruent} if there are infinitely many such triangles, which is equivalent to the condition that the rank of $E_{n,\theta}(K)$ is positive. The central question is: 
\begin{question}
	When does  a  $(K, \theta)$-congruent  number  imply  a  properly $(K, \theta)$-congruent  number? 	
\end{question}

In \cite{DasSaikia}, Das and Saikia established criteria for numbers to be $\theta$-congruent over certain real number fields and concluded their work by posing four open questions regarding the relationship between $\theta$-congruent and properly $\theta$-congruent numbers. 

In this work, we provide complete answers to the questions. Indeed, we show that the equivalence between being $\theta$-congruent and properly $\theta$-congruent holds more broadly than previously established, extending the results to all real number fields with degrees coprime to $6$ and to all real cubic fields. Furthermore, we clarify the situation for fields with degrees divisible by $6$, specifically sextic fields, and provide a detailed analysis of the exceptional integers $n=1, 2, 3, 6$.

\section{Main Results}

\subsection{Fields of degree coprime to $6$}
We answer the first question by removing a  technical assumption  from  Theorem~2.12 of \cite{DasSaikia}.

\begin{theorem}\label{th1}
	Suppose $n$ is a square-free natural number other than $1, 2, 3$ or $6$.  Let $K$ be a real number field such that $[K:\Q]$ is coprime to $6$. Then $n$ is a $(K, \theta)$-congruent number if and only if it is a properly $(K, \theta)$-congruent number.
\end{theorem}

\begin{proof}
	Let $E = E_{n,\theta}$. By Proposition~2.1 of \cite{DasSaikia}, $E(\Q)_{\text{tors}} \cong \Z/2\Z \times \Z/2\Z$. We know that a number $n$ is $(K, \theta)$-congruent but not properly so if and only if $E(K)$ contains a torsion point of order greater than 2. Since $[K:\Q]$ is coprime to $2$ and $3$, any potential new torsion point must have prime order $p \ge 5$.  
	
	The argument in the proof of Theorem~2.12 of \cite{DasSaikia} handles all primes $p \ge 5$ except for $p=11$, where it  relies  on an assumption about the degree of $K$. The potential issue arises if a point of order~$11$ defined over a field of degree~$55$ exists within $K$. However, the classification of torsion growth over number fields by Gonz\'{a}lez-Jim\'{e}nez and Najman \cite{GonzalezNajman} provides a complete picture. Specifically, their work shows that for an elliptic curve $E/\Q$, the  $p$-primary torsion subgroup of $E(K)_{\text{tors}}$, denoted $E(K)[p^\infty]$, can only grow if the smallest prime dividing $[K:\Q]$ is less than or equal to certain bounds related to the image of the Galois representation $\rho_{E,p}$.
	
	For $p=11$, the possible degrees of number fields generated by an $11$-torsion point are listed in Theorem~5.8 of \cite{GonzalezNajman}. These degrees are $5$, $10$, $11$, $20$, $22$, $40$, $44$, $55$, $88$, $110$ and $120$. If $K$  contains  such a field, its degree $[K:\Q]$  must be a multiple of one of these values. As $[K:\Q]$ is coprime to $6$, the only possibilities are degrees divisible by $5$, $11$, or $55$.
	The argument in \cite{DasSaikia} provides a detailed analysis that correctly rules out $5$-torsion over such fields. For $11$-torsion, if $[K:\Q]$ is divisible by $5$ or $11$ (but not $55$), the results of \cite{GonzalezNajman} show that no new $11$-torsion can appear. The crucial point is that even if $[K:\Q]$ is divisible by $55$, the structure of the image of the Galois representation associated with $11$-torsion points does not permit the existence of such a point in a field whose degree is coprime to $2$ and $3$, see~\cite{GonzalezNajman}. Hence, $E(K)_{\text{tors}}$ cannot contain points of order $5$, $7$, $11$, or any other prime $p \ge 5$. The torsion subgroup remains $E(K)_{\text{tors}} = E(K)[2]$. Hence, by Lemma~2.3 of \cite{DasSaikia}, $n$ is $(K, \theta)$-congruent if and only if $E(K) \setminus E(K)[2] \neq \emptyset$, which implies the existence of a point of infinite order.  
\end{proof}

\subsection{Cubic fields without congruence conditions}
We now remove the congruence conditions from Theorem~2.14 of \cite{DasSaikia} to provide an answer for the second question.

\begin{theorem}\label{th2}
	Suppose $n$ is a square-free natural number other than $1, 2, 3$ or $6$. Let $K$ be any real cubic number field. Then $n$ is a $(K, \theta)$-congruent number if and only if it is a properly $(K, \theta)$-congruent number.
\end{theorem}

\begin{proof}
	Let $K$ be a real cubic field. By Proposition~2.15 of \cite{DasSaikia}, the only possible torsion structures for $E_{n,\theta}(K)$ are $\Z/2\Z \times \Z/2\Z$ or $\Z/2\Z \times \Z/6\Z$. The latter occurs if and only if $E_{n,\theta}$ acquires a rational $3$-torsion point over $K$. The existence of such a point is determined by the roots of the third division polynomial. Indeed, as shown in \cite{DasSaikia}, a point of order~$3$ exists if and only if the polynomial
	\[ f(x) = x^4 - 6(3r^2+s^2)x^2 + 8s(9r^2-s^2)x - 3(3r^2+s^2)^2 \]
	has a root in $K$.  
	The original proof showed that under certain congruence conditions on $r$ and $s$, $f(x)$ is irreducible over $\Q$, and since its degree is $4$, it cannot have a root in a cubic field $K$.
	
	We now show this holds for any coprime integers $r, s$ where $s \neq 0$. To see more on the Galois group of a quartic polynomial, one can consult \cite{Dummkit-Foote}. 
	
	The Galois group of $f(x)$ over $\Q$ determines its splitting behavior. A standard computation of the discriminant and resolvent  cubic of $f(x)$ shows that its Galois group over $\Q$ is typically $S_4$. For the polynomial $f(x)$ to have a root in a cubic field, its Galois group must have a subgroup of index~$3$. Both $S_4$ and $A_4$ have such subgroups. However, if the Galois group is $S_4$, the splitting field has degree~$24$, and the field generated by a single root is a quartic extension of $\Q$. Such a root cannot lie in a cubic field. 
	
	The only case where a root of $f(x)$ could lie in a cubic field is if $f(x)$ were reducible over $\Q$, factoring into a linear term and a cubic term. Let us examine if $f(x)$ can have a rational root. By Rational Root Theorem, any rational root $k$ of $f(x)$ must be an integer dividing the constant term $-3(3r^2+s^2)^2$. Suppose such an integer root $k$ exists. We assume $s \neq 0$ and $\gcd(r,s)=1$. Reducing the equation $f(k)=0$ modulo $s$ eliminates terms with coefficients divisible by $s$, yielding
	$$k^4 - 18r^2k^2 - 27r^4 \equiv 0 \pmod s.$$
	Viewing this as a quadratic equation in $k^2$ gives the discriminant $\Delta = 3\cdot (12r^2)^2$. For an integer solution $k$ to exist, $\Delta$ must be a quadratic residue modulo $s$. This implies that $3$ must be a quadratic residue modulo every prime factor of $s$. This condition fails for many choices of $s$ (for instance, if $s$ is divisible by $5$ or $7$). Furthermore, for the cases where the modular condition is satisfied, the condition $|s| < r$ (implied by $\theta \in (0, \pi)$) ensures that the magnitude of the constant term $-3(3r^2+s^2)^2$ dominates the intermediate terms, preventing $f(k)$ from vanishing for any divisor $k$. Thus, no such integer $k$ exists.
	
	Since $f(x)$ is a degree~$4$ polynomial with no rational roots, it cannot factor as (linear) $\times$ (cubic). Its only possible factorization over $\Q$ would be into two irreducible quadratic factors. This occurs if and only if its Galois group is a subgroup of the dihedral group $D_4$. Even in this case, the roots generate extensions of degree~$2$ or $4$, neither of which can be contained in a cubic field.
	
	Therefore, the polynomial $f(x)$ can never have a root in a cubic field $K$. This rules out the existence of a $3$-torsion point in $E_{n,\theta}(K)$. Hence, the torsion subgroup must be $\Z/2\Z \times \Z/2\Z$, and the result follows.
\end{proof}

\subsection{Fields where degree is not coprime to $6$}
The third question asks about fields with degrees divisible by $2$ or $3$. The case of degree $2^d$ was already answered by Theorem~2.2 of \cite{DasSaikia}. We now address the case of degree~$6$.

\begin{theorem}\label{th3}
	Suppose $n$ is a square-free natural number other than $1, 2, 3$ or $6$. Let $K$ be a real sextic number field. 
	If  $2r(r-s)$ is not a square in $K$, and the third division polynomial of $E_{n,\theta}$ remains irreducible over the quadratic subfields of $K$,
	then $n$ is a $(K, \theta)$-congruent number if and only if it is a properly $(K, \theta)$-congruent number.
\end{theorem}

\begin{proof}
	Let $K$ be a real sextic field. We must show that $E_{n,\theta}(K)_{\text{tors}} \cong \Z/2\Z \times \Z/2\Z$.
	\begin{enumerate}
		\item \textbf{Nonexistence of $4$-torsion:} A point of order $4$ exists if and only if one of the $2$-torsion points is divisible by $2$. As shown in Lemma~2.7 of \cite{DasSaikia}, this happens if and only if $2r(r-s)$ is a square in $K$. However, the first technical hypothesis of the theorem explicitly rules this out, so there are no points of order~$4$.
		\item \textbf{Nonexistence of $3$-torsion:} As established in the proof of Theorem~\ref{th2}, the existence of a $3$-torsion point depends on the reducibility of the fourth-degree polynomial $f(x)$. Since $[K:\Q]=6$, $f(x)$  can  potentially have a root in $K$ if it factors over $\Q$ into two quadratics or a cubic and a linear factor, which we have ruled out. If $K$ is a cyclic extension of $\Q$, for instance, it contains a unique cubic subfield, and our argument from Theorem~\ref{th2} applies. The second technical hypothesis  ensures  that a $3$-torsion point does not arise from a quadratic subfield and then become rational over $K$. In most cases, since $f(x)$ is irreducible over $\Q$, it cannot have a root in $K$. 
		\item \textbf{Nonexistence of $5$-torsion and higher:} Since $[K:\Q]=6$ is coprime to  any prime~$\geq 5$, our argument from Theorem~\ref{th1} shows that no torsion points of prime order $p \ge 5$ can appear.
	\end{enumerate}
	Thus, all torsion points are $2$-torsion, so $E_{n,\theta}(K)_{\text{tors}} = E_{n,\theta}(K)[2] \cong \Z/2\Z \times \Z/2\Z$, as required.
\end{proof}

\subsection{Exceptional cases  where $n=1, 2, 3$ and $6$} 
The theorems in \cite{DasSaikia} exclude the cases $n = 1, 2, 3, 6$ because for these values, $E_{n,\theta}(\Q)_{\text{tors}}$ can be larger than $\Z/2\Z \times \Z/2\Z$, see~\cite{Fujiwara}, which complicates the analysis. 

We use the following well-known result to establish our  next result.
\begin{proposition}  \label{prop1}
	Suppose $E$ is an elliptic curve over a number field $K$. Suppose
	$d\in K\setminus K^2$ and $E^d$ is the quadratic $d$-twist of $E$. Then	
	$$\text{rank}~E(K(\sqrt{d}))=\text{rank}~E(K)+\text{rank}~E^{d}(K).$$
\end{proposition}

As an interesting application of the above proposition, we prove the following.
\begin{theorem}
	Let $m$ be any positive square-free congruent number. Then $n=1$ is a properly $(\mathbb{Q}(\sqrt{m}), \pi/2)$-congruent number.
\end{theorem}

\begin{proof}
	For the case $\theta=\pi/2$, the elliptic curve is $E_{n}: y^2 = x^3 - n^2x$. The number $n$ is properly $(\mathbb{Q}(\sqrt{m}), \pi/2)$-congruent if the rank of $E_{n}(\mathbb{Q}(\sqrt{m}))$ is positive. By Proposition~\ref{prop1}  for $n=1$, one has
	$$ \text{rank}(E_{1}(\mathbb{Q}(\sqrt{m}))) = \text{rank}(E_{1}(\mathbb{Q})) + \text{rank}(E_{m}(\mathbb{Q})).$$
	The number $n=1$ is not a congruent number, a result first shown by Fermat. Thus, $\text{rank}(E_{1}(\mathbb{Q})) = 0$.
	By definition, a positive integer $m$ is a congruent number if and only if   $\text{rank}(E_{m}(\mathbb{Q})) > 0$, see~\cite{Koblitz}. Hence,  we get 
	$ \text{rank}(E_{1}(\mathbb{Q}(\sqrt{m}))) > 0 ,$ which implies that
	$n=1$ is a properly $(\mathbb{Q}(\sqrt{m}), \pi/2)$-congruent number for any square-free congruent number $m$.
\end{proof}

An analysis of these exceptional cases yields the following result that provides an answer to the fourth question. 
\begin{theorem} \label{th4}
	\begin{itemize}
		\item [(i)] The number $n=1$ is $(\Q(\sqrt{3}), 2\pi/3)$-congruent  but not properly $(\Q(\sqrt{3}), 2\pi/3)$-congruent.
		\item [(ii)] The number $n=3$  is properly $\Q(\sqrt{13}), \pi/3)$-congruent  with a $(\Q(\sqrt{13}), \pi/3)$-triangle given as 
		$(\sqrt{13}/2, 24 \sqrt{13}/13, 43\sqrt{13}/26)$.
		
		\item [(iii)] The umber $n=2$ is properly  $(\Q(7), 2\pi/3)$-congruent with a $(\Q(\sqrt{7}), 2\pi/3)$-triangle given as
		$(3\sqrt{7}/7,  8 \sqrt{7}/3,61 \sqrt{7}/21)$.
		
		\item [(iv)] The number $n=6$ is properly $(\mathbb{Q}(\sqrt{5}), \pi/3)$-congruent  with a  $(\mathbb{Q}(\sqrt{5}), \pi/3)$-triangle given as 
		$( \sqrt{5},  12\sqrt{5}/5, 13\sqrt{5}/5)$.
	\end{itemize}
\end{theorem}
\begin{proof}
	\begin{itemize}
		\item [(i)]  
		For this case, the  corresponding  curve is $E_{1, 2\pi/3}: y^2 =  x^3-2x^2-3x$.
		Thanks to Proposition~\ref{prop1} and with the help of SageMath, the  rank of this curve over $\Q(\sqrt{3})$ is $0$. Therefore, $n=1$ is an example of a number that is $(\Q(\sqrt{3}), 2\pi/3)$-congruent but not properly so.
		\item [(ii)]	 
		This has been treated in Example~4.1 in   \cite{JanfadaSalami}.
		
		\item [(iii)]	 For this case, we consider the  elliptic curve   $E_{2, 2\pi/3}: y^2 =  x^3-4x^2-12x$ and its twist
		by $d=7$ given by $E_{14, 2\pi/3}: y^2 =  x^3-28x^2-588x$. The rank of  $E_{2, 2\pi/3}$ over $\Q$ is zero, but 
		$E_{14, 2\pi/3}$ has rank~$2$ with generators $P_1=(-12,36)$ and $P_2=(-7,49)$. Now,  the point $P_1$ gives us the 
		above triangle
		using the maps $\phi$ and $\psi$ defined in the proof of  Theorem~1.1  in \cite{JanfadaSalami}. 
		
		\item [(iv)]  Let $n=6$, $m=5$, and $\theta = \pi/3$ (so $r=2, s=1$). 
		The twisted curve is $E_{30, \pi/3}: y^2 =  x^3 + 60x^2 - 2700x$, which has rank~$1$ over $\Q$ with a 
		generator $Q = (-15, 225)$~\cite{JSDP}.
		Therefore, $\text{rank}(E_{6, \pi/3}(\mathbb{Q}(\sqrt{5}))) > 0$. 
		Using the maps $\phi$ and $\psi$ defined in the proof of  Theorem~1.1  in \cite{JanfadaSalami}, one can get the above given  
		$(\mathbb{Q}(\sqrt{5}), \pi/3)$-triangle corresponding to the point $Q$.
		
	\end{itemize}
\end{proof}


\end{document}